\documentclass[a4paper,11pt]{amsart}
\usepackage[utf8]{inputenc}
\usepackage{amsxtra}
\usepackage{amsopn}
\usepackage{amsmath,amsthm,amssymb}
\usepackage{amscd}
\usepackage{amsfonts}
\usepackage{latexsym}
\usepackage{verbatim}
\usepackage{color}

\theoremstyle{plain}
\newtheorem{theorem}{Theorem}[section]
\newtheorem*{theorem*}{Theorem}
\newtheorem{definition}[theorem]{Definition}
\newtheorem{lemma}[theorem]{Lemma}
\newtheorem{prop}[theorem]{Proposition}
\newtheorem{cor}[theorem]{Corollary}
\newtheorem{rem}[theorem]{Remark}
\newtheorem{ex}[theorem]{Example}
\newtheorem*{mt*}{Main Theorem}
\newcommand\C{{\mathbb C}}

\newcommand\R{{\mathbb R}}
\newcommand\Z{{\mathbb Z}}
\newcommand\T{{\mathbb T}}

\newcommand{\del}{\partial}
\newcommand{\delbar}{\overline{\del}}

\newcommand\ID{{\hbox{Id}}}

\title{Families of almost complex structures and transverse $(p,p)$-forms}
\author{Richard Hind, Costantino Medori and Adriano Tomassini}
\address{Department of Mathematics \\
University of Notre Dame \\
Notre Dame, IN 46556}
\email{hind.1@nd.edu}
\address{Dipartimento di Scienze Matematiche, Fisiche e Informatiche\\
Unit\`a di Matematica e Informatica,
Universit\`{a} degli Studi di Parma\\
Parco Area delle Scienze 53/A, 43124 \\
Parma, Italy}
\email{costantino.medori@unipr.it}
\email{adriano.tomassini@unipr.it}

\keywords{almost $p$-K\"ahler manifold; almost complex deformation; semi-K\"ahler metric}
\thanks{Partially supported by {\em Fondazione Bruno Kessler-CIRM (Trento)}\newline 
The first author is partially supported by Simons Foundation grant \# 633715. \newline The second and the third authors are partially supported by the Project PRIN 2017 ``Real and Complex Manifolds: Topology, Geometry and holomorphic dynamics''
and by GNSAGA of INdAM}
\subjclass[2010]{53C55, 53C25}
\begin{document}
\begin{abstract}
An {\em almost p-K\"ahler manifold} is a triple $(M,J,\Omega)$, where $(M,J)$ is an almost complex manifold of real dimension $2n$ and $\Omega$ is a closed real tranverse $(p,p)$-form on $(M,J)$, where 
$1\leq p\leq n$. When $J$ is integrable, almost $p$-K\"ahler manifolds are called $p$-{\em K\"ahler manifolds}. We produce families of almost $p$-K\"ahler structures $(J_t,\Omega_t)$ on $\C^3$, $\C^4$, and on the real torus $\mathbb{T}^6$, arising as deformations of K\"ahler structures $(J_0,g_0,\omega_0)$, such that the almost complex structures  $J_t$ cannot be locally compatible with any symplectic form for $t\neq 0$. Furthermore, examples of special compact nilmanifolds with and without almost $p$-K\"ahler structures are presented.

\end{abstract}
\maketitle
\tableofcontents

\section{Introduction}
Let $(M,J,g,\omega)$ be a $2n$-dimensional compact K\"ahler manifold, that is $M$ is a compact $2n$-dimensional smooth manifold endowed with an integrable almost complex structure $J$ and $J$-Hermitian metric $g$ whose fundamental form $\omega$ is closed, namely $(J,g,\omega)$ is a K\"ahler structure on $M$. Then, the celebrated theorem of Kodaira and Spencer \cite[Theorem 15]{KS} states that the K\"ahler condition is stable under small $\mathcal{C}^\infty$-deformations of the complex structure $J$. As a consequence of complex Hodge Theory, the existence of a K\"ahler structure on a compact manifold imposes strong restrictions on the topology of $M$, e.g., the odd index Betti numbers of $M$ are even, the even index Betti numbers are greater than zero and the de Rham complex of $M$ is a formal differential graded algebra in the sense of Sullivan. In \cite{HL} Harvey and Lawson give an intrinsic characterization in terms of currents of compact 
complex manifolds admitting a K\"ahler metric. However there are many examples of compact complex manifolds without any K\"ahler structures; examples are  easily constructed by taking compact quotients of simply connected nilpotent Lie groups. The underlying differentiable manifolds of such complex manifolds may carry further structures, e.g., symplectic structures, and {\em balanced}, {\em SKT}, {\em Astheno K\"ahler} metrics, that is Hermitian metrics whose fundamental form $\omega$ satisfies $d\omega^{n-1}=0$, respectively $\del\delbar\omega=0$, respectively $\del\delbar\omega^{n-2}=0$, where $\dim_\C M=n$. 

An almost K\"ahler manifold is an almost complex manifold equipped with a Hermitian metric whose fundamental form is closed. 
In the almost K\"ahler setting stability properties are drastically different. First of all, a dense subset of almost complex structures $J$ on $\R^{2n}$, with $n>2$, are not compatible with any symplectic form, that is, there are no symplectic forms $\omega$, such that $g_J(\cdot,\cdot)=\omega(\cdot,J\cdot)$ is a $J$-Hermitian metric, or equivalently, $\omega$ is a closed positive $(1,1)$-form. Such almost complex structures can be extended to any K\"ahler manifold $(M,J)$ of dimension bigger than $2$, showing that there exists a curve $\{J_t\}_{t\in(-\varepsilon,\varepsilon)}$ such that $J_0=J$ and $J_t$, for $t\neq 0$, is a (non-integrable) almost complex structure on $M$, which is not even locally compatible with respect to any symplectic form. This is a direct consequence of e.g., \cite[Theorem 2.4, Corollary 2.5]{MT} which deals with the local case. 

In the present paper, starting with a complex manifold $(M,J)$, we are interested 
in studying the stability properties of transversely closed $(p,p)$-forms, namely the stability properties of $p$-{\em K\"ahler manifolds} in the terminology of Alessandrini and Andreatta \cite{AA} under possible non integrable small deformations of the complex structure.\newline 
The notion of transversality was firstly introduced by Sullivan in \cite{S}, in the context of {\em cone structures}, namely a continuous 
field of cones of {\em p}-vectors on a manifold, and then the $p$-K\"ahler condition was studied by several authors (see e.g., \cite{AA,AB1} and the references therein). In particular, $1$-K\"ahler manifolds correspond to K\"ahler manifolds and $(n-1)$-K\"ahler manifolds correspond  to {\em balanced} manifolds in the terminology of Michelsohn (see \cite{M}): for the proofs of these results see  \cite[Proposition 1.15]{AA} or \cite[Corollary 4.6]{RWZ1}. K\"{a}hler manifolds are $p$-K\"{a}hler for all $p$ (by taking the $p$-th power of the form) but  $p$-K\"{a}hler manifolds may not be K\"ahler. Note that there is a difference between the case $p=n-1$ and $p<n-1$. Indeed, according to \cite[p.279]{M}, if $\Omega$ is an $(n-1)$-K\"ahler structure, then $\Omega=\omega^{n-1}$ for a suitable fundamental form $\omega$ of a Hermitian metric on $M$. Hence $d\omega^{n-1}=0$ and so $M$ admits a balanced metric. On the other hand if $p<n-1$ and $\omega$ is the fundamental form of an almost Hermitian metric on $M$, then $d\omega^p=0$ implies that $d\omega=0$, so in fact $M$ is almost K\"ahler (see \cite[Theorem 3.2]{GH}). 

In contrast to the K\"ahler case, the $p$-K\"ahler condition on a compact complex manifold is not stable under small deformation of the complex structure. This was proved in \cite{AB1} by constructing a non balanced deformation of the natural complex structure for the Iwasawa manifold, which carries a balanced metric. 

Recently, in \cite{RWZ}, Rao, Wan and Zhao further studied the stability of $p$-K\"ahler compact manifolds under small integrable deformations of the complex structure. Assuming that the $(p, p + 1)$-th mild $\del\delbar$-lemma holds, it is shown that  $p$-K\"ahler structures are stable for all $1 \le p \le n-1$. Here the $(p, p + 1)$-th mild $\del\delbar$-lemma for a complex manifold means that each $\del$-exact and $\delbar$-closed
$(p, p + 1)$-form on this manifold is $\del\delbar$-exact. Note that such a condition does not hold in the Iwasawa example in \cite[p.1062]{AB1}. For other recent results on families of compact balanced manifolds see \cite{Sf}.\newline
In the terminology by Gray and Hervella \cite[p.40]{GH}, a $J$-Hermitian metric $g$ on an almost complex manifold $(M,J)$ of complex dimension $n$ is said to be {\em semi-K\"ahler} if $d\omega^{n-1}=0$.\smallskip

The aim of this paper is to produce families of almost $p$-K\"ahler structures $(J_t,\Omega_t)$ arising as deformations of K\"ahler structures $(J_0,g_0,\omega_0)$, such that the almost complex structures  $J_t$ cannot be locally compatible with any symplectic form for $t\neq 0$.

Our first main result is the following, see Theorem \ref{main-theorem-1} \medskip

\noindent{\bf Theorem} {\em Let $(J,\omega)$ be the standard K\"ahler structure on the standard torus 
$\T^{6}=\R^{6}\slash \Z^{6}$, with coordinates $(z_1,z_2,z_3)$, $z_j=x_j+iy_j$, $j=1,2,3$. Let $f=f(z_2,\overline{z}_2)$ be a $\Z^{6}$-periodic smooth complex valued function on $\R^6$ and set $f=u+iv$, where $u=u(x_2,y_2)$, $v=v(x_2,y_2)$. \newline
Let $I=(-\varepsilon,\varepsilon)$. Assume that $u$, $v$ satisfy the following condition
 $$
 \left(\frac{\partial u}{\partial x_2}
 +
 \frac{\partial v}{\partial y_2}\right)\Big\vert_{(x,y)=0}\neq 0
$$
Then, for $\varepsilon >0$ small enough, there exists a $1$-parameter complex family of almost 
complex structures $\{J_{t}\}_{t\in I}$ on $\mathbb{T}^{6}$, such that
\begin{enumerate}
 \item[I)] $J_{0}=J$;
 \item[II)] $J_{t}$ admits a semi-K\"ahler metric for all $t\in I$;
 \item[III)] For any given $0\neq t\in I $, the almost complex structure $J_{t}$ is not locally compatible with respect to any symplectic form on $\mathbb{T}^6$. 
\end{enumerate}
} 
\medskip





Then, starting with a balanced structure on $(J,g,\omega)$ on $\C^n$, we provide necessary conditions in order that a curve $(J_t,g_t,\omega_t)$ give rise to semi-K\"ahler structures on $\C^n$ (Theorem \ref{m-theorem-prime}). Special results are obtained for $n=3$, $p=2$ (Theorem \ref{m-theorem}, Corollaries \ref{m-cor-3} and \ref{m-cor-3-first}).

Finally we have a result for almost $2$-K\"{a}hler structures on $\C^4$, see Theorem \ref{C-4-main}.

\noindent{\bf Theorem} {\em There exists a family of almost-complex structures $J_t$ on $\C^4$ such that $J_0=i$ is the standard integrable complex structure and $J_t$ is almost $2$-K\"{a}hler for all $t \neq 0$, but $J_t$ is not locally K\"{a}hler for all $t \neq 0$.
} 
\medskip

It is interesting to contrast this with the linear case, where we see in Proposition \ref{linearcase} that a complex structure preserving a product $\omega^p$ automatically preserves $\omega$ (up  to sign). Hence for $p<n$ the almost $p$-K\"ahler (but not almost K\"{a}hler) forms in a deformation cannot be powers of 
$2$-forms $\omega_t$ (since the $\omega_t$ would necessarily be closed and hence almost K\"{a}hler). In particular it is impossible to find a deformation as in the theorem above where the almost $2$-K\"{a}hler forms remain equal to $\omega_0^p$ (where $\omega_0$ is the standard K"{a}hler form on $\C^4$). This differs from the K\"{a}hler case, where by Moser's theorem we may apply a family of diffeomorphisms $\Phi_t$ to any deformation $(J_t,\omega_t)$, with $[\omega_t]$ constant, such that the form is unchanged, that is $(\Phi_t^*J_t,\Phi_t^*\omega_t)=(J'_t,\omega_0)$. \smallskip

The paper is organized as follows: in Section \ref{preliminaries} we start by fixing notation and recalling some basic facts on $p$-K\"ahler and almost $p$-K\"ahler structures, giving a simple example of a compact $(2n-1)$-dimensional complex manifold without any $p$-K\"ahler structure, for $1\leq p\leq (n-1)$. In Section \ref{complex-curves} we prove Proposition \ref{linearcase}. In Proposition \ref{fixedomega} we provide an example of a deformation of a non-K\"{a}hler integrable complex structure into nonintegrable structures such that none of the structures are almost K\"{a}hler, in fact they cannot be tamed by any symplectic form, but they are all almost $2$-K\"{a}hler, and in fact all compatible with the same $(2,2)$ form. Finally Sections \ref{main} and \ref{main-examples} are devoted to the proofs of the main results and to the constructions of the almost $p$-K\"ahler families. In particular, in Theorem \ref{C-4-main} we obtain the family of almost $2$-K\"ahler structures in $\C^4$ which are not almost K\"ahler, and in Proposition \ref{prop-Iwasawa} we construct a curve of almost complex structures $\{J_t\}_{t\in\R}$ (non integrable for $t\neq0$) on the Iwasawa manifold such that $J_0$ admits a balanced metric and $J_t$ does not admit any semi-K\"ahler metric for $t\neq 0$.

\vskip.2truecm\noindent
\noindent {\em Acknowledgements.} We would like to thank {\em Fondazione Bruno Kessler-CIRM (Trento)} for their support and very pleasant working environment. We would like also to thank S. Rao and Q. Zhao for useful comments and remarks.
\smallskip

\section{Almost $p$-K\"ahler structures}\label{preliminaries}
Let $V$ be a real $2n$-dimensional vector space endowed with a complex structure $J$, that is an automorphism $J$ of $V$ satisfying $J^2=-\hbox{\rm id}_V$. Let $V^*$ be the dual space of $V$ and denote by the same symbol the complex structure on $V^*$ naturally induced by $J$ on $V$. Then the complexified 
$V^{*\C}$ decomposes as the direct sum of the $\pm\,i$-eigenspaces, $V^{1,0}$, $V^{0,1}$ of the extension of $J$ to $V^{*\C}$, given by
$$
\begin{array}{l}
V^{1,0}=\{\varphi\in V^{*\C}\,\,\,\vert\,\,\,J\varphi=i\varphi\}= 
\{\alpha-iJ\alpha \,\,\,\vert\,\,\,\alpha\in V^{*}\}\\
V^{0,1}=\{\psi\in V^{*\C}\,\,\,\vert\,\,\,J\psi=-i\psi\}= 
\{\beta+iJ\beta \,\,\,\vert\,\,\,\beta\in V^{*}\},
 \end{array}
$$
that is 
$$V^{*\C}=V^{1,0}\oplus V^{0,1}.
$$
According to the above decomposition, the space $\Lambda^r(V^{*\C})$ of complex $r$-covectors on $V^\C$ 
decomposes as 
$$
\Lambda^r(V^\C)=\bigoplus_{p+q=r}\Lambda^{p,q}(V^{*\C}),
$$
where 
$$
\Lambda^{p,q}(V^{*\C})=\Lambda^p(V^{1,0})\otimes\Lambda^q(V^{0,1}).
$$
If $\{\varphi^1,\ldots,\varphi^{n}\}$ is a basis of $V^{1,0}$, then
$$
\{\varphi^{i_1}\wedge\cdots\wedge\varphi^{i_p}\wedge\overline{\varphi^{j_1}}\wedge\cdots\wedge\overline{\varphi^{j_q}}\,\,\,\vert\,\,\, 1\leq i_1<\cdots<i_p\leq n,\,\,
1\leq j_1<\cdots<j_q\leq n\}
$$
is a basis of $\Lambda^{p,q}(V^{*\C})$. Set $\sigma=i^{p^2}2^{-p}$. Then, given any 
$\varphi\in\Lambda^{p,0}(V^{*\C})$ we have that
$$
\overline{\sigma_p\varphi\wedge\overline{\varphi}}=\sigma_p\varphi\wedge\overline{\varphi},
$$
that is $\sigma_p\varphi\wedge\overline{\varphi}$ is a $(p,p)$-real form. Consequently, denoting by
$$
\Lambda_{\R}^{p,p}(V^{*\C})=\{\psi\in\Lambda^{p,p}(V^{*\C})\,\,\,\vert\,\,\,\psi=\overline{\psi}\},
$$
we get that 
$$
\{\sigma_p\varphi^{i_1}\wedge\cdots\wedge\varphi^{i_p}\wedge\overline{\varphi^{i_1}}\wedge\cdots\wedge\overline{\varphi^{i_p}}\,\,\,\vert\,\,\, 1\leq i_1<\cdots<i_p\leq n\}
$$
is a basis of $\Lambda_{\R}^{p,p}(V^{*\C})$.
\begin{rem}
The complex structure $J$ acts on the space of real $k$-covectors $\Lambda^k(V^*)$ by setting, for any given $\alpha\in \Lambda^k(V^*)$,
$$J\alpha (V_1,\ldots,V_k)=\alpha(JV_1,\ldots,JV_k).
$$
Then it is immediate to check that if $\psi\in\Lambda_{\R}^{p,p}(V^{*\C})$ then $J\psi=\psi$. For $k=2$, the converse holds.
\end{rem}
Denoting by 
$$
\hbox{\rm Vol}=(\frac{i}{2}\varphi^1\wedge\overline{\varphi^1})\wedge\cdots\wedge
(\frac{i}{2}\varphi^n\wedge\overline{\varphi^n}),
$$
we obtain that 

$$
\hbox{\rm Vol}=\sigma_n\varphi^1\wedge\cdots \wedge\varphi^n\wedge \overline{\varphi^1}\wedge\cdots\wedge
\wedge\overline{\varphi^n},
$$
that is $\hbox{\rm Vol}$ is a volume form on $V$. A real $(n,n)$-form $\psi$ is said to be {\em positive} respectively {\em strictly positive} if 
$$\psi=a{\rm Vol},
$$ 
where $a\geq 0$, respectively $a>0$. By definition, $\psi\in\Lambda^{p,0}(V^{*\C})$ is said to be {\em simple} or {\em decomposable} if 
$$
\psi=\eta^1\wedge\cdots\wedge\eta^p,
$$
for suitable $\eta^1,\ldots,\eta^p\in V^{1,0}$. Let $\Omega\in\Lambda_{\R}^{p,p}(V^{*\C})$. Then $\Omega$ is said to be {\em transverse} if, given any non-zero simple $(n-p)$-covector $\psi$, the real $(n,n)$-form 
$$
\Omega\wedge\sigma_{n-p}\psi\wedge \overline{\psi}
$$
is strictly positive. \newline
The notion of positivity on complex vector spaces can be transferred pointwise to almost complex 
manifolds. Let $(M,J)$ be an almost complex manifold of real dimension $2n$; we will denote by 
$A^{p,q}(M)$ the space of complex $(p,q)$-forms, that is the space of smooth sections of the bundle 
$\Lambda^{p,q}(M)$ and by $A_{\R}^{p,p}(M)$ the space of real $(p,p)$-forms.
\begin{definition}
 Let $(M,J)$ be an almost complex manifold of real dimension $2n$ and let $1\leq p\leq n$. A 
 $p$-{\em K\"ahler form} is a closed real transverse $(p,p)$-form $\Omega$, that is 
 $\Omega$ is $d$-closed and, at every $x\in M$, $\Omega_x\in\Lambda^{p,p}_{\R}(T_x^*M)$ is transverse. The triple $(M,J,\Omega)$ is said to be an {\em almost p-K\"ahler manifold}.
\end{definition}

\section{Curves of almost complex structures preserving the $2^{th}$-power}\label{complex-curves}
Let $I=(-\varepsilon,\varepsilon)$ and 
$\{J_t\}_{t\in I}$ be a smooth curve of almost
complex structures on $M$, such that $J_0=J$. Then, for small $\varepsilon$, there exists a unique $L_t:TM\to TM$, 
with $L_tJ+JL_t=0$, for every $t\in I$, such that
\begin{equation}
J_t=(\ID+L_t)J(\ID+L_t)^{-1}, 
\end{equation}
for every $t$ (see e.g., \cite{AL}, \cite[Sec. 3]{dBM}). We can write $L_t=tL+o(t)$. Assume that $J$ is {\em compatible} with respect to a symplectic form $\omega$ on $M$, that is, at any given $x\in M$, 
$$
g_x(\cdot,\cdot):=\omega_x(\cdot, J\cdot)
$$
is a positive definite Hermitian metric on $M$. Equivalently, $\omega$ is a positive $(1,1)$-form with respect to $J$. Then 
$\omega$ is a positive $(1,1)$-form with respect to $J_t$ if and only if $L_t$ is $g_{J}$-symmetric and 
$\parallel L_t\parallel < 1$. \newline 
We can show the following
\begin{prop}\label{linearcase}
 Let $\omega$ be a positive $(1,1)$-form on the complex vector space $\C^n$. Let $\{J_t\}$ be a curve of linear complex structures on $\C^n$. If $J_t$ preserves $\omega^p$ for all $t$, $1\leq p <n$, then $J_t$ preserves $\omega$.
\end{prop}
\begin{proof}
We will show that a complex structure on $\C^n$ preserving $\omega^p$ preserves $\omega$ up to sign. Hence in a $1$-parameter family with $J_0 =i$ all complex structures must preserve $\omega$ itself.

First note that a subspace $W \subset \C^n$ of real dimension at least $2p$ is symplectic if and only if $\omega^p$ is nondegenerate. Also, for a symplectic subspace of dimension at least $2p$ the symplectic complement can be defined either in the usual way as $$W^{\perp} := \{ v \in \C^n | \omega(v,w)=0 \mbox{ for all } w \in W\}$$ or equivalently as $$W^{\perp} := \{ v \in \C^n | \omega^p(v,w_1, \dots, w_{2p-1})=0 \mbox{ for all } w_1, \dots, w_{2p-1} \in W\}.$$

Suppose then that $J$ is a complex structure preserving $\omega^p$ and $V$ is a $2$-dimensional symplectic plane. Then $V^{\perp} = W$ is symplectic and hence $\omega^p|_W$ is nondegenerate. As $J$ preserves $\omega^p$ we have that $\omega^p|_{JW}$ is also nondegenerate and therefore also symplectic, and using the definition of $(JW)^{\perp}$ only in terms of $\omega^p$ we see that $(JW)^{\perp}=JV$. Hence $JV$ is also a symplectic plane. Moreover, if $U$ is a symplectic plane in $V^{\perp}$ then $JU \subset J(V^{\perp}) = JW = (JV)^{\perp}$.

Let $x_1, y_1, \dots , x_n, y_n$ be a basis of $\C^n$ with $\omega(x_k, y_k)=1$ for all $k$ and such that the symplectic planes $V_k = \mbox{Span}(x_k, y_k)$ are orthogonal (for example we can take the standard symplectic basis). Then by the remark at the end of the previous paragraph the planes $JV_k$ are also symplectically orthogonal. Set $\lambda_k =  \omega(Jx_k, Jy_k)$. Since $J$ preserves $\omega^p$ we have that $\lambda_{i_1} \lambda_{i_2} \dots \lambda_{i_p} =1$ for all $1 \le i_1 < \dots < i_p \le n$. Hence either all $\lambda_k =1$ or all $\lambda_k = -1$. (The second case is only possible when $p$ is even.) It follows that $J$ is either symplectic or anti-symplectic.
 
\end{proof}
Notice that there exist exact $2$-K\"ahler structures on $3$-dimensional compact complex manifolds, that is balanced metrics $g$, such that $\omega^2$ is $d$-exact, where $\omega$ denotes the fundamental form of $g$. This is in contrast with the almost K\"ahler case, in view of Stokes Theorem. 

\begin{ex}{\em 
Let $G=SL(2,\C)$. Then $G$ admits compact quotients by uniform discrete subgroups $\Gamma$, so that 
$$
M=\Gamma\backslash G
$$
is a $3$-dimensional compact complex manifold. Denote by 
$$
Z_1=\frac12
\left[
\begin{array}{cc}
i & 0 \\ 
0 & -i
\end{array}
\right],\quad
Z_2=\frac12
\left[
\begin{array}{cc}
0 & -1 \\ 
1 & 0
\end{array}
\right]
\quad
Z_3=\frac12
\left[
\begin{array}{cc}
0 & i \\ 
i & 0
\end{array}
\right];
$$
then $\{Z_1,Z_2,Z_3\}$ is a basis of the Lie algebra of $G$. We have
$$
[Z_1,Z_2]=-Z_3,\quad [Z_1,Z_3]=Z_2,\quad [Z_2,Z_3]=-Z_1.
$$
Accordingly, the dual left-invariant coframe $\{\psi^1,\psi^2,\psi^3\}$ satisfies the Maurer-Cartan equations
$$
d\psi^1=\psi^2\wedge \psi^3,\quad d\psi^2=-\psi^1\wedge \psi^3,\quad d\psi^3=\psi^1\wedge \psi^2.
$$
Let 
$$
\Omega =\frac14(\psi^{12\bar{1}\bar{2}}+\psi^{13\bar{1}\bar{3}}+\psi^{23\bar{2}\bar{3}}),
$$
where we indicated $\psi^{r\bar{s}}=\psi^r\wedge\overline{\psi^s}$. Then,
\begin{equation}\label{Omega-exact}
\Omega=\frac18 d(\psi^{12\bar{3}}+\psi^{\bar{1}\bar{2}3}-
\psi^{13\bar{2}}-\psi^{\bar{1}\bar{3}2}+
\psi^{23\bar{1}}+\psi^{\bar{2}\bar{3}1}
):=d\gamma
\end{equation}
In \cite[p.467]{OUV} it is proved that every left-invariant $(2,2)$-form is $d$-exact (see also \cite{PT} for cohomological computations).
Let us define a complex curve of almost complex structures on $\Gamma\backslash G$ through a basis of $(1,0)$ forms by setting
$$
\psi_t^1=\psi^1,\qquad 
\psi_t^2=\psi^2,\qquad 
\psi_t^3=\psi^3-t\overline{\psi^3},\qquad 
$$
for $t\in \mathbb{B} (0,\varepsilon)$. A direct computation gives
\begin{equation}\label{mc-deformation}
\left\{
\begin{array}{lll}
d\psi^1_t&=&\frac{1}{1-\vert t\vert ^2}(\psi^{23}_t+t\psi^{2\bar{3}}_t)\\[5pt]
d\psi^2_t&=&-\frac{1}{1-\vert t\vert ^2}(\psi^{13}_t+t\psi^{1\bar{3}}_t)\\[5pt]
d\psi^3_t&=&\psi^{12}_t-t\psi^{\bar{1}\bar{2}}_t.
\end{array}
\right.
\end{equation}
In particular, the last equation shows that $J_t$ is not integrable for $t\neq 0$. Now we compute the action of $J_t$ on the forms $\psi^1,\psi^2,\psi^3$. A straightforward calculation yields to
$$\
J_t\psi^1=i\psi^1,\quad J_t\psi^2=i\psi^2,\quad
J_t\psi^3=\frac{i}{1-\vert t\vert^2}\big((1+\vert t\vert^2)\psi^3-2t\overline{\psi^3}\big)
$$
and 
$$
J_t\overline{\psi^1}=-i\overline{\psi^1},\quad J_t\overline{\psi^2}=-i\overline{\psi^2},\quad
J_t\overline{\psi^3}=\frac{i}{1-\vert t\vert^2}\big(-(1+\vert t\vert^2)\overline{\psi^3}+2\overline{t}\psi^3\big)
$$
Therefore, it immediate to check that 
$$
J_t\Omega=\Omega,
$$
so that $\Omega$ is $(2,2)$-with respect to $J_t$ for every $t$. \newline
Finally, there are no symplectic structures taming $J_t$, for every given $t\in \mathbb{B}(0,\varepsilon)$. By contradiction: assume that there exist a symplectic structure $\omega_t$ on $\Gamma\backslash SL(2,\C)$ taming $J_t$. Then, since for every $t$ the almost complex structure is left invariant, by an average process, we may produce a left invariant symplectic structure $\hat{\omega}$ on $\Gamma\backslash SL(2,\C)$, 
taming $J_t$. Let $\hat{\omega}$ be given as
$$
2\hat{\omega}=iA\psi_t^{1\bar{1}}+iB\psi_t^{2\bar{2}}+iC\psi_t^{3\bar{3}}+u\psi_t^{1\bar{2}}-\bar{u}\psi_t^{2\bar{1}} +v\psi_t^{1\bar{3}}-\bar{v}\psi_t^{3\bar{1}} +w\psi_t^{2\bar{3}}-\bar{w}\psi_t^{3\bar{2}},
$$
where $A,B,C,u,v,w\in\C$. Then, a direct calculation using \eqref{mc-deformation} gives that, if $\hat{\omega}$ is closed, then $C=0$. This is absurd.}
\end{ex}
Therefore, we have proved the following
\begin{prop}\label{fixedomega}
For any given $t\in\mathbb{B}(0,\varepsilon)\subset\C$, $(\Gamma\backslash SL(2,\C),J_t,\Omega)$ is a compact almost $2$-K\"ahler manifold of complex dimension $3$, such that:
\begin{enumerate}
 \item[i)] the almost complex structure $J_t$ is integrable if and only if $t=0$;
 \item[ii)] the almost $2$-K\"ahler structure $\Omega$ is $d$-exact; 
 \item[iii)] $J_t$ has no tamed symplectic structures for every given $t\in \mathbb{B}(0,\varepsilon)$.
\end{enumerate}
\end{prop}
We end this Section proving a result of non-existence of almost $p$-K\"ahler structures.
\begin{prop}\label{nop} Let $(M,J)$ be a closed almost complex manifold of (complex) dimension n.

Suppose $\alpha$ is a non closed $1$-form such that the $(1,1)$ part
$$(d\alpha)^{1,1} = \sum_k c_k \psi_k \wedge \bar{\psi_k}$$
where the $\psi_k$ are $(1,0)$-covectors and the $c_k$ have the same sign. Then $(M,J)$ does not have a balanced metric.

More generally, suppose there exists a non closed $(2n-2p-1)$-form $\beta$ such that
$$(d\beta)^{n-p,n-p} = \sum_k c_k \psi_k \wedge \bar{\psi_k}$$
where the $\psi_k$ are simple $(n-p,0)$-covectors and the $c_k$ have the same sign. Then $(M,J)$ does not admit an almost $p$-K\"{a}hler form.
\end{prop}

\begin{proof} It suffices to prove the second statement. Without loss of generality, we may assume that all the $c_k$ are positive. Arguing by contradiction, suppose that $\Omega$ is a $p$-K\"{a}hler form and $\beta$ is a non-zero $(2n-2p-1)$-form as above. Then since $M$ is closed we have
$$0 = \int_M \sigma_{n-p} d(\Omega \wedge \beta) = \sum_k c_k \int_M \Omega \wedge \sigma_{n-p} \psi_k \wedge \bar{\psi_k} >0$$
as all integrals on the right are strictly positive. This gives a contradiction.
\end{proof}

\begin{rem} All Riemann surfaces are almost K\"{a}hler. Therefore if a $1$-form $\alpha$ as in Proposition \ref{nop} exists, it must restrict to a closed form on all $1$-dimensional subvarieties of $M$.
\end{rem}
As an application of Proposition \ref{nop}, we provide a family of $n$-dimensional compact complex manifolds which are not $p$-K\"ahler, for any given $1\leq p\leq (n-1)$. 
\begin{ex}{\em 
Let 
$$
\mathbb{H}_{2n-1}(\R):=\Big\{
A=\left [
\begin{array}{lll}
1 & X & v\\
0 & I_{n-1} & Y\\
0 & 0 &1
\end{array}
\right]\,\,\,\vert\,\,\, X^t,Y\in\R^{n-1},\,v\in\R
\Big\}
$$
be the $(2n-1)$-dimensional real Heisenberg group, where $I_{n-1}$ denotes the identity matrix of order $n-1$. Then $\mathbb{H}_{2n-1}(\R)$ is $(2n-1)$-dimensional nilpotent Lie group and the subset
$\Gamma\subset\mathbb{H}_{2n-1}(\R)$, formed by matrices having integers entries, is a uniform discrete subgroup of $\mathbb{H}_{2n-1}(\R)$, so that $\Gamma\backslash\mathbb{H}_{2n-1}(\R)$ is a comapct $(2n-1)$-dimensional nilmanifold. Then 
$$
M=\Gamma\backslash\mathbb{H}_{2n-1}(\R)\times \R\slash\Z
$$
is a $2n$-dimensional compact nilmanifold having a global coframe $\{e^1,\ldots,e^n,f^1,\ldots,f^n\}$, defined as 
$$
\begin{array}{lll}
e^{\alpha}=dx^\alpha,&1\leq\alpha \leq n-1, & e^n=du\\
f^{\alpha}=dy^\alpha,&1\leq\alpha \leq n-1, & f^n=dv-\displaystyle\sum_{\beta=1}^{n-1}x^\beta dy^\beta ,
\end{array}
$$
where $u$ denotes the natural coordinate on $\R$. It is immediate to check that
\begin{equation}
\left\{
\begin{array}{ll}
de^{\alpha}=0 & 1\leq\alpha \leq n\\[5pt]
df^{\beta}=0 & 1\leq\beta \leq n-1\\[5pt]
df^{n}=-\displaystyle\sum_{\gamma =1}^{n-1}e^\gamma\wedge f^\gamma & 
\end{array}
\right. 
\end{equation}
Then 
$$
\varphi^\alpha=e^\alpha +if^\alpha,\qquad 1\leq\alpha\leq n
$$
give rise to a complex coframe of $(1,0)$-forms on $M$ such that 
$$
\left\{
\begin{array}{ll}
d\varphi^{\alpha}=0 & 1\leq\alpha \leq n-1\\[5pt]
d\varphi^n=\displaystyle\frac12\sum_{\beta =1}^{n-1}\varphi^\beta\wedge\overline{\varphi^\beta} & 
\end{array}
\right.
$$
so that induced almost complex structure $J$ is integrable,  
Fix any $1\leq p\leq (n-1)$. We show that $(M,J)$ is not $p$-K\"ahler. Define
$$
\beta=\varphi^n\wedge\varphi^{1\overline{1}\ldots (n-p-1)\overline{(n-p-1)}}.
$$
Then,
$$
d\beta=(d\beta)^{n-p,n-p}=\Big(\sum_{\beta=1}^{n-1}\varphi^{\beta\overline{\beta}}\Big)
\wedge\varphi^{1\overline{1}\ldots (n-p-1)\overline{(n-p-1)}}
$$
and the result follows from Proposition \ref{nop}.
\newline 
Note that the smooth manifold $M$ does not carry any K\"ahler structure $(J,g,\omega)$, since it is a non toral nilmanifold.
}
\end{ex}

\section{Semi-K\"ahler deformations of balanced metrics}
\label{main}
Let $M=\C^n$ endowed with a balanced structure $(J,g,\omega)$, that is $J$ is a complex structure on $\C^n$, $g$ is a Hermitian metric such that the fundamental form $\omega$ of $g$ satisfies $d\omega^{n-1}=0$.
Denote by $\{\varphi^1,\ldots,\varphi^{n}\}$ be a complex $(1,0)$-coframe on $(\C^n,J)$, and let  
$$
\omega=\frac{i}{2}\sum_{j,k=1}^n\omega_{jk}(z)\varphi^j\wedge\overline{\varphi^k},
$$
be the fundamental form of $g$, where $(\omega_{jk}(z))$ is Hermitian and positive definite.
Then, assuming $\omega$ is a balanced metric,
$$\Omega:=\frac{1}{(n-1)!}\omega^{n-1}$$
is $d$-closed. By assumption $d\Omega=0$. Denote by $\{\zeta_1,\ldots,\zeta_n\}$ the dual $(1,0)$-frame of $\{\varphi^1,\ldots,\varphi^{n}\}$. Let $I=(-\varepsilon,\varepsilon)$ and let $\{J_t\}_{t\in I}$ be a smooth curve of almost complex structures on $\C^n$ such that $J_0=J$.
Then, as already recalled in Section \ref{complex-curves}, there exists a unique $L_t\in\hbox{\rm End}(T\C^n)$ such that $L_tJ+JL_t=0$ and 
$$
J_t=(I+L_t)J(I+L_t)^{-1}.
$$
Then, $L_t$ can be identified with an element $\Phi(t)\in \Gamma(\C^n,\Lambda^{0,1}\C^n\otimes T^{1,0}\C^n)$, by defining
$$
\Phi(t)=\frac12(L_t-iL_t).
$$
In other words, the curve of almost complex 
structures $\{J_t\}_{t\in I}$ is encoded by such a $\Phi(t)\in \Gamma(\C^n,\Lambda^{0,1}\C^n\otimes T^{1,0}\C^n)$ so that, if the expression of $\Phi(t)$ is  
$$
\Phi(t)=\sum_{h,k=1}^n\sigma^j_k(z,t)\overline{\varphi^k}\otimes\zeta_j,
$$
with $\sigma^j_k=\sigma^j_k(z,t)$ smooth on $(z,t)$, then 
a complex $(1,0)$-coframe on $(\C^n,J_t)$ is given by 
$$
\varphi^j_t=\varphi^j-\langle \Phi(t),\zeta_j\rangle,\quad j=1,\ldots,n,
$$
where 
$$
\langle \Phi(t),\zeta_j\rangle=\sum_{k=1}^n\sigma^j_k\overline{\varphi^k}
$$
Then, explicitly
\begin{equation}\label{def-phi-t}
\varphi^j_t=\varphi^j-\sum_{k=1}^n\sigma^j_k\overline{\varphi^k},\quad \quad j=1,\ldots,n.
\end{equation}
According to the Kodaira and Spencer theory of small deformations of complex structures (see \cite{MK}, \cite{GHJ}) $J_t$ is integrable if and only if the Maurer-Cartan equation holds, that is 
$$
\delbar\Phi(t)+\frac12[[\Phi(t),\Phi(t)]]=0.
$$
Here, we are not assuming that $J_t$ is integrable. Let $(J_t,g_t,\omega_t)$ be a curve of almost Hermitian metrics on $\C^n$ such that $(J_0,g_0,\omega_0)=(J,g,\omega)$. Then,
\begin{equation}\label{def-omega-t-prime}
\omega_t:=\frac{i}{2}\sum_{j,k=1}^n\omega_{jk}(z,t)\varphi_t^j\wedge\overline{\varphi_t^k};
\end{equation}
where $\omega_{jk}(z,t)$ are smooth and $\omega_{jk}(z,0)=\omega_{jk}(z)$, $j,k=1,\ldots, n$. 
\newline
In the sequel we will use the symbol $\dot{}$ to denote the derivative with respect to $t$, e.g., we will use the following notation
$$
\dot{\varphi}^j_0=\frac{d}{dt}{\varphi}^j_t\vert_{t=0}
$$
and we will drop $0$ for the $t$ derivative of the functions $\sigma^j_k(z_1,\ldots,z_n,t)$ evaluated at $t=0$, that is 
$$
\dot{\sigma^j_k}=\frac{d}{dt}
\sigma^j_k(z_1,\ldots,z_n,t)\vert_{t=0}.
$$
Set 
\begin{equation}\label{def-Omega-t}
\Omega_t:=\frac{1}{(n-1)!}\omega_t^{n-1}.
\end{equation}
 Let us compute the $t$ derivative of $\Omega_t$ at $t=0$. 
In view of \eqref{def-phi-t}, we easily compute
$$
\dot{\varphi}^j_0=-\sum_{k=1}^n\dot{\sigma^j_k}\overline{\varphi^k},
$$
by the definition of $\Omega_t$, we obtain that
$$\dot{\Omega}_0\in A^{n,n-2}\C^n\oplus A^{n-1,n-1}\C^n\oplus A^{n-2,n}\C^n,\qquad \dot{\Omega}_0=\overline{\dot{\Omega}}_0.
$$
Consequently, we can define the $(n-2,n)$-form $\eta$ and the $(n-1,n-1)$-form $\lambda$ on $(\C^n,J_t)$ respectively by
\begin{equation}\label{derivative-Omega-t-prime}
\eta:=(\dot{\Omega}_0)^{n-2,n} 
\end{equation}
and 
\begin{equation}\label{derivative-n-1-Omega-t-prime}
\lambda:=(\dot{\Omega}_0)^{n-1,n-1}. 
\end{equation}
We have
\begin{equation}\label{Omega-eta-prime}
\dot{\Omega}_0=\eta+\overline{\eta}+\lambda.
\end{equation}
We are ready to state the following
\begin{theorem}\label{m-theorem-prime}
Let $(J,g,\omega)$ be a balanced structure on $\C^n$. Let $(J_t,g_t,\omega_t)$, for $t\in I$, be a curve of almost Hermitian metrics on $\C^n$ such that $(J_0,g_0,\omega_0)=(J,g,\omega)$. If $(J_t,g_t,\omega_t)$ is a curve of semi-K\"ahler structures on $\C^n$, then 
\begin{equation}\label{del-eta-prime}
\del\eta+\delbar\lambda=0.
\end{equation}
\end{theorem}
\begin{proof}
By definition, $\Omega_t:=\frac{1}{(n-1)!}\omega_t^{n-1}$ and by assumption, 
\begin{equation}\label{d-Omega-t-prime}
d\Omega_t=0.
\end{equation}
Thus, by taking the derivative of \eqref{d-Omega-t-prime} with respect to $t$ evaluated at $t=0$ and taking into account \eqref{Omega-eta-prime}, we obtain
$$
0=d\dot{\Omega}_0=d(\eta+\overline{\eta}+\lambda)=(\del+\delbar)(\eta+\overline{\eta}+\lambda)=\del\eta+\delbar\overline{\eta}+\del\lambda+\delbar\lambda,
$$
where we have used that $J_t$ is integrable for $t=0$ and that $\eta\in A^{n-2,n}\C^n$. Therefore, the above equation
$$
\del\eta+\delbar\overline{\eta}+\del\lambda+\delbar\lambda=0
$$
turns to be equivalent, by type reasons, to 
$$
\del\eta+\delbar\lambda=0,
$$
that is, if $\Omega_t$ is $d$-closed, then \eqref{del-eta-prime} holds. The Theorem is proved.
\end{proof}
Let $M$ be a compact holomorphically parallelizable complex manifold. Then, by a Theorem of Wang (see \cite{W}), there exists a simply-connected, connected complex Lie group $G$ and a lattice $\Gamma\subset G$ such that 
$M=\Gamma\backslash G$. Assume that $M$ is a solvmanifold, that is $G$ is solvable. Then, according to Nakamura \cite[Prop. 1.4]{N}, the universal covering of $M$ is biholomorphically equivalent to $\C^N$. Due to Abbena and Grassi \cite[Theorem 3.5]{AG}, the natural complex structure $J$ on $M$ admits a balanced metric $g$. 
As a direct consequence of Theorem \ref{m-theorem-prime}, we get the following (see also Sferruzza \cite{Sf})
\begin{cor}\label{solvable-necessary.condition}
Let $M=\Gamma\backslash\C^N$ be a compact complex solvmanifold endowed with the natural balanced structure $(J,g,\omega)$. If $(J_t,g_t,\omega_t)$ is a curve of semi-K\"ahler structures on $M$ such 
that $(J_0,g_0,\omega_0)=(J,g,\omega)$, then 
\begin{equation}\label{cohomological-prime}
0=[\del\eta]_{\delbar}\in H_{\delbar}^{N-1,N}(M).
\end{equation}
\end{cor}
Coming back to $\C^n$, in the particular case that   
\begin{equation}\label{def-omega-t}
\omega_t:=\frac{i}{2}\sum_{j=1}^n\varphi_t^j\wedge\overline{\varphi_t^j},
\end{equation}
we obtain the following 
\begin{prop}\label{m-theorem}
Let $(J,g,\omega)$ be a balanced structure on $\C^n$. Let $\{J_t\}_{t\in I}$ be a smooth curve of almost complex structures on $\C^n$ such that $J_0=J$. Let $\omega_t$ be the real $J_t$-positive $(1,1)$-form defined by \eqref{def-omega-t} and $g_t$ be the associated $J_t$-Hermitian metric on $\C^n$. If $(J_t,g_t,\omega_t)$ is a curve of semi-K\"ahler structures on $\C^n$, then 
\begin{equation}\label{del-eta}
\del\eta=0.
\end{equation}
\end{prop}
The proof of the above Proposition follows at once from Theorem \ref{m-theorem-prime} by noting that in such a case $\lambda=0$.\newline 
Finally, under the same assumptions as in the last Proposition \ref{m-theorem}, for $n=3$, we derive the following
\begin{cor}\label{m-cor-3}
Let $(J,g,\omega)$ be a balanced structure on $\C^3$. Let $\{J_t\}_{t\in I}$ be a smooth curve of almost complex structures on $\C^3$ such that $J_0=J$. Let $\omega_t$ be the real $J_t$-positive $(1,1)$-form defined by \eqref{def-omega-t} and $g_t$ be the associated $J_t$-Hermitian metric on $\C^3$. If $(J_t,g_t,\omega_t)$ is a curve of semi-K\"ahler structures on $\C^3$, then 
\begin{equation}\label{del-eta-3}
\del \left(\left[(\dot{\sigma}^2_3-\dot{\sigma}^3_2)\varphi^1 +(\dot{\sigma}^3_1-\dot{\sigma}^1_3)\varphi^2+(\dot{\sigma}^1_2-\dot{\sigma}^2_1)\varphi^3 \right]\wedge\varphi^{\bar{1}\bar{2}\bar{3}}\right)=0.
\end{equation} 
\end{cor}
\begin{cor}\label{m-cor-3-first}
In the same assumptions as in the Corollary \ref{m-cor-3}, under the additinal assumption that $\varphi^j=dz_j, j=1,2,3$, if $(J_t,g_t,\omega_t)$ is a curve of semi-K\"ahler structures on $\C^3$, then
the following equations hold
\begin{equation}\label{necessary-condition-3}
\left\{
\begin{array}{l}
\frac{\partial}{\partial z_1}(\dot{\sigma}^3_1-\dot{\sigma}^1_3)-\frac{\partial}{\partial z_2}(\dot{\sigma}^2_3-\dot{\sigma}^3_2)=0\\[10pt] 
\frac{\partial}{\partial z_1}(\dot{\sigma}^1_2-\dot{\sigma}^2_1)-\frac{\partial}{\partial z_2}(\dot{\sigma}^2_3-\dot{\sigma}^3_2)=0\\[10pt]
\frac{\partial}{\partial z_2}(\dot{\sigma}^1_2-\dot{\sigma}^2_1)-\frac{\partial}{\partial z_3}(\dot{\sigma}^3_1-\dot{\sigma}^1_3)=0
\end{array}
\right. 
\end{equation}
\end{cor}
\begin{rem}
The previous constructions can be easily adapted replacing the parameter space $I=(-\varepsilon,\varepsilon)$, with $$\mathbb{B}(0,\varepsilon)=\{t=(t_1,\ldots,t_k\}\in\R^k\,\,\,:\,\,\, \vert t\vert <\varepsilon.$$ 
In particular, Theorem \ref{m-theorem-prime} can be generalized by considering
$$
\eta_j=\left(\frac{\partial}{\partial t_j}\Omega\vert_{t=0}\right)^{n-2,n},\quad
\lambda_j=\left(\frac{\partial}{\partial t_j}\Omega\vert_{t=0}\right)^{n-1,n-1}.
$$
\end{rem}

\section{Applications and examples}\label{main-examples}
First, we construct a family of semi-K\"ahler structures on the $6$-dimensional torus $\mathbb{T}^6$, obtained as a deformation of the standard K\"ahler structure on $\mathbb{T}^6$, which cannot be locally compatible with any symplectic form. More precisely, we start by showing the following 
\begin{theorem}\label{main-theorem-1}
Let $(J,\omega)$ be the standard K\"ahler structure on the standard torus 
$\T^{6}=\R^{6}\slash \Z^{6}$, with coordinates $(z_1,z_2,z_3)$, $z_j=x_j+iy_j$, $j=1,2,3$. Let $f=f(z_2,\overline{z}_2)$ be a $\Z^{6}$-periodic smooth complex valued function on $\R^6$ and set $f=u+iv$, where $u=u(x_2,y_2)$, $v=v(x_2,y_2)$. \newline
There is an almost complex structure $J = J(f)$ on $\mathbb{T}^6$ such that
\begin{enumerate}
 \item[I)] If $f=0$ then $J=J(0)$ is the standard complex structure;
 \item[II)] $J$ admits a semi-K\"ahler metric provided $f$ is sufficiently small (in the uniform norm);
 \item[III)] Suppose that $u$, $v$ satisfy the following condition
\begin{equation}\label{condition-Derivative}
 \left(\frac{\partial u}{\partial x_2}
 +
 \frac{\partial v}{\partial y_2}\right)\Big\vert_{(x,y)=0}\neq 0.
\end{equation}
Then $J$ is not locally compatible with respect to any symplectic form on $\mathbb{T}^6$. 
\end{enumerate}
\end{theorem}
\begin{proof}
 I) Define
 \begin{equation}\label{almost-complex-structure-torus}
\left\{
\begin{array}{ll}
J\partial_{x_1}=& 2v\partial_{x_3}+\partial_{y_1}-2u\partial_{y_3}\\[5pt]
J\partial_{x_2}=& \partial_{y_2}\\[5pt]
J\partial_{x_3}=& \partial_{y_3}\\[5pt]
J\partial_{y_1}=& -\partial_{x_1}-2u\partial_{x_3}-2v\partial_{y_3}\\[5pt]
J\partial_{y_2}=& -\partial_{x_1}\\[5pt]
J\partial_{y_3}=& -\partial_{x_3}\\[5pt]
\end{array}
\right.
\end{equation}
By definition of $J$, in view of the $\Z^6$-periodicity of the functions $u$, $v$ we immediately get that $J$ gives rise to an almost complex strcuture on $\mathbb{T}^6$, such that $J(0)$  coincides with the standard complex structure on $\mathbb{T}^6$.\vskip.2truecm\noindent
 II) The almost complex structure defined by \eqref{almost-complex-structure-torus} induces an almost complex structure on the cotangent bundle of $\mathbb{T}^6$, still denoted by $J$ and   expressed as
 \begin{equation}\label{almost-complex-structure-dual-torus}
\left\{
\begin{array}{ll}
J dx_1=& -dy_1\\[5pt]
J dx_2=& -dy_2\\[5pt]
J dx_3=&-dy_3+ 2vdx_1+\partial_{y_1}-2udy_1\\[5pt]
J dy_1=& dx_1\\[5pt]
J dy_2=& dx_2\\[5pt]
J dy_3=&dx_3- 2udx_1-2vdy_1
\end{array}
\right.
\end{equation}

Accordingly, a $(1,0)$-coframe on $\mathbb{T}^6$ with respect to $J$ is given by
$$
\varphi^1=dz_1,\quad \varphi^2=dz_2,\quad\varphi^3=dz_3-fd\bar{z}_1.
$$
Thus, 
$$
\sigma^3_1=f(z_2,\overline{z}_2),
$$
with the other $\sigma^j_k$ vanishing. Then, for $f$ small enough, $J$ admits a semi-K\"ahler metric $g$, whose fundamental form is provided by
$$
\omega=\frac{i}{2}\sum_{j=1}^3\varphi^j\wedge\overline{\varphi^j}
$$
\vskip.2truecm\noindent
III) In order to show that the almost complex structure admits no locally compatible symplectic forms when \eqref{condition-Derivative}  holds, we need to recall the following result.\newline 
Let $P$ be an almost complex structure on $\R^6$, with coordinates $(x_1,\ldots,x_6)$. Then, according to \cite[Theorem 2.4]{MT}, if $P$ is locally compatible with respect to a symplectic form, then the following necessary conditions hold

\begin{equation}\label{equations-mt}
\left\{
\begin{array}{l}
-\frac{\partial}{\partial x_1}(P_{26}-P_{62})-\frac{\partial}{\partial x_2}(P_{16}-P_{61})
-\frac{\partial}{\partial x_3}(P_{15}-P_{51})\\[10pt]
-\frac{\partial}{\partial x_4}(P_{23}-P_{32})+\frac{\partial}{\partial x_5}(P_{13}-P_{31})
-\frac{\partial}{\partial x_6}(P_{12}-P_{21})=0\\[15pt]
-\frac{\partial}{\partial x_1}(P_{23}-P_{32})-\frac{\partial}{\partial x_2}(P_{13}-P_{31})
-\frac{\partial}{\partial x_3}(P_{12}-P_{21})\\[10pt]
-\frac{\partial}{\partial x_4}(P_{26}-P_{62})+\frac{\partial}{\partial x_5}(P_{16}-P_{61})
-\frac{\partial}{\partial x_6}(P_{15}-P_{51})=0,
\end{array}
\right.
\end{equation}
where all the derivatives are computed at $x=0$. In our notation, $x_4=y_1$, $x_5=y_2$, $x_6=y_3$.
\vskip.2truecm\noindent
Now, by the assumption \eqref{condition-Derivative} 
$$
\left(\frac{\partial u}{\partial x_2}
 +
 \frac{\partial v}{\partial y_2}\right)\Big\vert_{(x,y)=0}\neq 0
$$ 
on the partial derivatives at $(x,y)=0$ of $f=u+iv=u(x_2,y_2)+iv(x_2,y_2)$, so  we immediately see that the first equation of \eqref{equations-mt} is not satisfied. Therefore, $J(f)$ cannot be locally compatible with any symplectic form.
\end{proof}
\begin{rem} We can generate $1$-parameter families of almost complex structures by setting $J_t = J(tf)$. These show that the necessary condition \eqref{necessary-condition-3} of Corollary \ref{m-cor-3-first} is also sufficient in this case. Indeed, in such a case the natural $
\omega_t=\frac{i}{2}\sum_{j=1}^3\varphi^j_t\wedge\overline{\varphi^j_t}
$ satisfies $d\omega^2_t=0$.\newline
As a consequence of III), for any given $t\neq 0$, $J_t$ is not integrable. 
\end{rem}
\begin{ex}{\em (Iwasawa manifold)\,}{\em Let $\C^3$ be endowed with the product $*$ defined by
$$
(w_1,w_2,w_3)*(w_1,w_2,w_3)=(w_1+z_1,w_2+z_2,w_3+w_1z_2+z_3).
$$
Then $(\C^3,*)$ is a complex nilpotent Lie group which admits a lattice $\Gamma =\Z[i]^3$ and accordingly it turns out that
$\mathbb{I}_3=\Gamma\backslash\C^3$ is a compact complex $3$-dimensional manifold, the {\em Iwasawa manifold}. It is immediate to check that
$$
\{\varphi^1=dz_1,\quad \varphi^2=dz_2,\quad \varphi^3=dz_3-z_1dz_2\}
$$
is a complex $(1,0)$-coframe for the standard complex structure naturally induced by $\C^3$, whose dual frame is
$$
\{\zeta_1=\frac{\partial}{\partial z_1},\quad 
\zeta_2=\frac{\partial}{\partial z_2}+z_1\frac{\partial}{\partial z_3},\quad 
\zeta_3=\frac{\partial}{\partial z_3}\}.
$$
The following
$$
g=\sum_{j=1}^3\varphi^j\odot\overline{\varphi^j}
$$
is a balanced metric on $\mathbb{I}_3$. Indeed,
the fundamental form of $g$ is
$$
\omega=\frac{i}{2}(dz_1\wedge d\overline{z_1}+dz_2\wedge d\overline{z_2}-z_1dz_2\wedge d\overline{z_3}-\overline{z_1}dz_3\wedge d\overline{z_2}+dz_3\wedge d\overline{z_3},
$$
which satisfies $d\omega^2=0$. \newline
We will provide a smooth curve of almost complex structures $\{J_t\}_{t\in I}$ such that $J_0$ coincides with the complex structure on $\mathbb{I}_3$ and such that, for any $t\neq 0$, $J_t$ admits no  semi-K\"ahler metric. To this purpose let us define $\{J_t\}_{t\in I}$ on $\mathbb{I}_3$ by assigning 
$$
\Phi(t)=t\overline{\varphi^1}\otimes\zeta_2-
t\overline{\varphi^2}\otimes\zeta_1.
$$
Then, accordingly,
$$
\{\varphi^1_t=dz_1+td\overline{z_2},\quad \varphi^2_t=dz_2-td\overline{z_1},\quad \varphi^3_t=dz_3-z_1dz_2\}
$$
is a complex $(1,0)$-coframe on $(\mathbb{I}_3,J_t)$. A simple calculation yields to the following structure equations
\begin{equation}\label{structure-equations-iwasawa-deformed}
d\varphi^1_t=0,\quad
d\varphi^2_t=0,\quad
d\varphi^3_t=-\frac{1}{1+t^2}\left[
\varphi^{12}_t+t(\varphi_t^{1\bar{1}}+\varphi_t^{2\bar{2}}) +t^2\varphi_t^{\bar{1}\bar{2}}
\right]
\end{equation}
The last equation implies that $J_t$ is not integrable for $t\neq 0$.
First we observe that, by the definition of $\Phi(t)$,  
$$\sigma^2_1=t,\quad\sigma^1_2=-t,\qquad \sigma^j_k=0,\,\hbox{\rm otherwise}.
$$
Hence
\begin{eqnarray*}
 \del \left(\left[(\dot{\sigma}^2_3-\dot{\sigma}^3_2)\varphi^1 +(\dot{\sigma}^3_1-\dot{\sigma}^1_3)\varphi^2+(\dot{\sigma}^1_2-\dot{\sigma}^2_1)\varphi^3 \right]\wedge\varphi^{\bar{1}\bar{2}\bar{3}}\right)&=&2\del\varphi^{3\bar{1}\bar{2}\bar{3}}\\
 &=&-2\varphi^{12\bar{1}\bar{2}\bar{3}}
\end{eqnarray*}
A simple calculation shows that 
$$
0\neq [-2\varphi^{12\bar{1}\bar{2}\bar{3}}]_{\delbar}\in H_{\delbar}^{2,3}(\mathbb{I}_3),
$$
that is the necessary condition 
$$
0\neq [\del\eta]_{\delbar}\in H_{\delbar}^{2,3}(\mathbb{I}_3)
$$
of Corollary \ref{solvable-necessary.condition} is not satisfied.
Indeed, the form 
$$
-2\varphi^{12\bar{1}\bar{2}\bar{3}}
$$
is $\delbar$-harmonic with respect to the Hermitian 
metric 
$$
g=\sum_{j=1}^3\varphi^j\odot\overline{\varphi^j}
$$
on $\mathbb{I}_3$. In fact a stronger statement is true. We show that $J_t$ does not admit any semi-K\"ahler metric for $t\neq 0$. By the third equation of \eqref{structure-equations-iwasawa-deformed}, we immediately get that 
$$
\left(d\varphi^3_t\right)^{1,1}=-\frac{t}{1+t^2}\left(
\varphi_t^{1\bar{1}}+\varphi_t^{2\bar{2}}\right)
$$
Then, Proposition \ref{nop} applies, proving the following
\begin{prop}\label{prop-Iwasawa}
The curve $\{J_t\}_{t\in I}$ of almost complex structures on $\mathbb{I}_3$ satisfies the following:
\begin{enumerate}
 \item $J_0$ coincides with the natural complex structure on $\mathbb{I}_3$ and admits a balanced metric.
 \item For any given $t\neq 0$, $J_t$ is not integrable and it has no semi-K\"ahler metrics.
\end{enumerate}

\end{prop}
}
\end{ex}
\begin{ex}
{\em  (A family of almost $2$-K\"ahler structures on $\C^4$) Let $M=\C^4$ with real coordinates $(x_1,\ldots,x_4,y_1,\ldots,y_4)$ and $g$ be a smooth real valued function on $\C^4$. Define an almost complex structure $\mathcal{J}=\mathcal{J}_g$ on $\C^4$  by setting
\begin{equation}\label{almost-complex-structure-C4}
\left\{
\begin{array}{ll}
\mathcal{J}\partial_{x_1}=& g\partial_{x_3}+\partial_{y_1}\\[5pt]
\mathcal{J}\partial_{x_2}=& \partial_{y_2}\\[5pt]
\mathcal{J}\partial_{x_3}=& \partial_{y_3}\\[5pt]
\mathcal{J}\partial_{x_4}=& \partial_{y_4}\\[5pt]
\mathcal{J}\partial_{y_1}=& -\partial_{x_1}-g\partial_{y_3}\\[5pt]
\mathcal{J}\partial_{y_2}=& -\partial_{x_1}\\[5pt]
\mathcal{J}\partial_{y_3}=& -\partial_{x_3}\\[5pt]
\mathcal{J}\partial_{y_4}=& -\partial_{x_4}
\end{array}
\right.
\end{equation}
Then, a straightforward calculation shows that 
\begin{equation}\label{PHI-4}
\left\{
\begin{array}{lll}
\Phi^1&=&dx_1+idy_1,\\[5pt]
\Phi^2&=&dx_2+idy_2,\\[5pt]
\Phi^3&=&dx_3+i(-gdx_1+dy_3),\\[5pt]
\Phi^4&=&dx_4+idy_4
\end{array}
\right.
\end{equation}
is a complex $(1,0)$-coframe on $(\C^4,\mathcal{J})$. Define 
$$
\Omega=\frac{1}{4}\sum_{j<k}\Phi^{j\bar{j}k\bar{k}}
$$
and set
$$
\Omega_\tau=\Omega+\tau\hbox{\rm Re}\,\Phi^{2\bar{3}4\bar{4}}
$$
We have the following 
\begin{lemma}\label{lemma-C-4}
The form $\Omega_\tau$ is $(2,2)$ with respect to $\mathcal{J}$ and positive for every $\tau\in (-\varepsilon,\varepsilon)$, for $\varepsilon$ small enough. Futhermore, assume that $g=g(x_2,x_3)$. Then, for any given $\tau$, we have that 
$$
d\Omega_\tau=0 \qquad \iff \qquad
\frac{\partial g}{\partial x_2}-2\tau
\frac{\partial g}{\partial x_3}=0
$$
\end{lemma}
\begin{proof}
By \eqref{PHI-4} we obtain
$$
\begin{array}{ll}
\Phi^{1\bar{1}2\bar{2}}&=-4dx_1\wedge dy_1\wedge dx_2\wedge dy_2,\quad
\Phi^{1\bar{1}3\bar{3}}=-4dx_1\wedge dy_1\wedge dx_3\wedge dy_3\\[7pt]
\Phi^{1\bar{1}4\bar{4}}&=-4dx_1\wedge dy_1\wedge dx_4\wedge dy_4,\quad
\Phi^{2\bar{2}4\bar{4}}=-4dx_2\wedge dy_2\wedge dx_4\wedge dy_4\\[7pt]
\Phi^{2\bar{2}3\bar{3}}&=4\left(gdx_1\wedge dx_2\wedge dx_3\wedge dy_2-dx_2\wedge dy_2\wedge dx_3\wedge dy_3\right)\\[7pt]
\Phi^{3\bar{3}4\bar{4}}&=-4\left(gdx_1\wedge dx_3\wedge dx_4\wedge dy_4+dx_3\wedge dy_3\wedge dx_4\wedge dy_4\right)
\end{array}
$$
In view of the above formulae, since $g=g(x_2,x_3)$, we get
\begin{equation}\label{d-OMEGA}
d\Omega=\frac14 d\sum_{j<k}\Phi^{j\bar{j}k\bar{k}}=
\frac{\partial g}{\partial x_2}
dx_1\wedge dx_2\wedge dx_3\wedge dx_4\wedge dy_4
\end{equation}
Expanding $\Phi^{2\bar{3}4\bar{4}}$, we obtain
\begin{eqnarray*}
\hbox{\rm Re}\,\Phi^{2\bar{3}4\bar{4}}&=&
2\left(-gdx_1\wedge dx_2\wedge dx_4\wedge dy_4 
-dx_2\wedge dy_3\wedge dx_4\wedge dy_4
+\right.\\
&{}& \left.+dy_2\wedge dx_3\wedge dx_4\wedge dy_4
\right) 
\end{eqnarray*}
Consequently,
\begin{equation}\label{d-PHI-2344}
 d\hbox{\rm Re}\,\Phi^{2\bar{3}4\bar{4}}=
 -2\frac{\partial g}{\partial x_3}
dx_1\wedge dx_2\wedge dx_3\wedge dx_4\wedge dy_4
\end{equation}
Taking into account the defnition of 
$$
\Omega_\tau=\Omega +\tau\hbox{\rm Re}\,\Phi^{2\bar{3}4\bar{4}},
$$
the Lemma follows immediately from \eqref{d-OMEGA} and \eqref{d-PHI-2344}.
\end{proof}

We are ready to state and prove the following
\begin{theorem}\label{C-4-main}
Let $g$ be a smooth real valued function on $\C^4$ such that $g=g(x_2,x_3)$. Assume that, for any given $\tau\in (-\varepsilon,\varepsilon)$ the function $g$ satisfies the partial differential equation
$$
\frac{\partial g}{\partial x_2}-2\tau
\frac{\partial g}{\partial x_3}=0.
$$
Let $\mathcal{J}$ be the almost complex structure on $\C^4$ associated with $g$. Then,
\begin{enumerate}
 \item[i)] For any given $\tau\in (-\varepsilon,\varepsilon)$, $(\mathcal{J},\Omega_\tau)$ is an almost $2$-K\"ahler structure on $\C^4$.
 \item[ii)] If $\frac{\partial g}{\partial x_2}\vert_{(x,y)=0}\neq 0$, then $\mathcal{J}$ cannot be locally compatible with respect to any symplectic form on $\C^4$.
\end{enumerate}
\end{theorem}
\begin{proof}
i) By Lemma \ref{lemma-C-4} and by assumption on $g$, we immediately obtain that $(\mathcal{J},\Omega_\tau)$ is an almost $2$-K\"ahler structure on $\C^4$. \vskip.2truecm\noindent
ii) We will apply again \cite[Theorem 2.4]{MT}.\newline 
Let $\R^6\cong\C^3_{z_1,z_2,z_3}$, where $z_j=x_j+iy_j,\,j=1,2,3$. Then the restriction $\mathcal{J}\vert_{\R^6}$ of $\mathcal{J}$ to $\R^6$ gives rise to an almost complex structure on $\R^6$. We have that 
$$
\left\{
\begin{array}{ll}
\mathcal{J}\partial_{x_1}=& g\partial_{x_3}+\partial_{y_1}\\[5pt]
\mathcal{J}\partial_{x_2}=& \partial_{y_2}\\[5pt]
\mathcal{J}\partial_{x_3}=& \partial_{y_3}\\[5pt]
\mathcal{J}\partial_{y_1}=& -\partial_{x_1}-g\partial_{y_3}\\[5pt]
\mathcal{J}\partial_{y_2}=& -\partial_{x_1}\\[5pt]
\mathcal{J}\partial_{y_3}=& -\partial_{x_3}\\
\end{array}
\right.,
$$
where $g=g(x_2,x_3)$ satisfies 
$$
\frac{\partial g}{\partial x_2}-2\tau
\frac{\partial g}{\partial x_3}=0.
$$
By assumption, $
\frac{\partial g}{\partial x_2}\vert_{(x,y)=0}\neq 0$. Therefore, in view of \cite[Theorem 2.4]{MT}, the second equation of \eqref{equations-mt} is not satisfied. Consequently, $\mathcal{J}\vert_{\R^6}$ cannot be locally compatible with any symplectic form on $\R^6\cong\C^3_{z_1,z_2,z_3}$. \newline 
By contradiction: assume that there exists a symplectic form on $\C^4$ locally compatible with $\mathcal{J}$. Thus $(\mathcal{J},\omega)$ gives rise to an almost K\"ahler metric $\mathcal{G}$ on $\C^4$. Hence, $(\mathcal{J}\vert_{\R^6},\mathcal{G}\vert_{\R^6})$ would be an almost K\"ahler structure on the submanifold $\R^6\subset\C^4$. This is absurd.
\end{proof}
As an explicit example, we may take $g$ defined as
$$
g(x_2,x_3)=2\tau x_2+x_3.
$$
Then, 
\begin{itemize}
 \item[a)] $g$ verifies the partial differential equation
 $$
\frac{\partial g}{\partial x_2}-2\tau
\frac{\partial g}{\partial x_3}=0.
$$
\item[b)] For every $\tau\in (-\varepsilon,\varepsilon),$ $\tau\neq 0$, it is $\frac{\partial g}{\partial x_2}\vert_{(x,y)=0}\neq 0$.
\item[c)] For $\tau=0$ the almost complex structure is not integrable and it is almost K\"ahler.
\end{itemize}
Therefore, for any $\tau\in (-\varepsilon,\varepsilon),$ $\tau\neq 0$, $(\mathcal{J},\Omega_\tau)$ is a almost $2$ K\"ahler structure on $\C^4$, such that $\mathcal{J}$ admits no compatible symplectic structures.

}
\end{ex}

\end{document}